\documentclass[11pt]{amsart}
\usepackage{amsfonts}
\usepackage{graphicx,color}
\usepackage{amsthm}
\usepackage{amssymb,amsmath}

\title[Energy bounds in Hamming spaces]{\bf Energy bounds for codes and designs in Hamming spaces}

\def\ds{\displaystyle}
\date{\today}

\newtheorem{theorem}{Theorem}[section]
\newtheorem{lemma}[theorem]{Lemma}
\newtheorem{corollary}[theorem]{Corollary}

\theoremstyle{definition}
\newtheorem{defn}[theorem]{Definition}

\newtheorem{remark}[theorem]{Remark}
\newtheorem{problem}[theorem]{Problem}

\newcommand{\R}{\mathbb{R}}

\author[P. Boyvalenkov]{P. G. Boyvalenkov $^\dagger$}
\address{Institute of Mathematics and Informatics, Bulgarian Academy of Sciences,
8 G Bonchev Str.,
1113  Sofia, Bulgaria \\
and Faculty of Mathematics and Natural Sciences, South-Western University, Blagoevgrad, Bulgaria.
}
\email{peter@math.bas.bg}
\thanks{\noindent $^\dagger$ The research of this author was supported, in part, by a Bulgarian NSF contract I01/0003.}

\author[P. Dragnev]{P. D. Dragnev $^{\dagger \dagger}$}
\address{Department of Mathematical Sciences,
Indiana-Purdue University
Fort Wayne, IN 46805, USA }
\email{dragnevp@ipfw.edu}
\thanks{\noindent $^{\dagger \dagger}$ The research of this author was supported, in part, by a Simons Foundation grant no. 282207.}

\author[D. Hardin]{D. P. Hardin$^*$}
\address{Center for Constructive Approximation, Department of Mathematics, \hspace*{.1in}
Vanderbilt University,
Nashville, TN 37240, USA  }
\email{doug.hardin@vanderbilt.edu}

\author[E. Saff]{E. B. Saff$^*$}
\email{edward.b.saff@vanderbilt.edu}
\thanks{\noindent $^*$ The research of these authors was supported, in part,
by the U. S. National Science Foundation under grants DMS-1412428 and DMS-1516400.
}
\author[M. Stoyanova]{M. M. Stoyanova$^{**}$}
\address{Faculty of Mathematics and Informatics,
Sofia University,
5 James Bourchier Blvd.,
1164 Sofia, Bulgaria}
\email{stoyanova@fmi.uni-sofia.bg}
\thanks{
\noindent $^{**}$ The research of this author was supported, in part, by the
Science Foundation of Sofia University under contracts 144/2015 and 57/2016.
}
\thanks{The authors express their gratitude to Erwin Schr\"{o}dinger International Institute for providing
conducive research atmosphere during their stay when this manuscript was started.}

\def\ds{\displaystyle}

\begin{document}
\maketitle

\begin{abstract}
We obtain universal bounds on the energy of codes and designs in Hamming spaces.
Our bounds hold for a large class of potential functions, allow a unified treatment, and
can be viewed as a generalization of the Levenshtein bounds for maximal codes.
\end{abstract}



\section{Introduction}

Let $Q=\{0,1,\ldots,q-1\}$ be the alphabet of $q$ symbols and $\mathbb{H}(n,q)$   the set of all $q$-ary
vectors $x=(x_1,x_2,\ldots,x_n)$ over $Q$. The Hamming distance $d(x,y)$ between
points $x=(x_1,x_2,\ldots,x_n)$ and $y=(y_1,y_2,\ldots,y_n)$ from $\mathbb{H}(n,q)$ is equal to the number of
coordinates in which they differ. The use of $q$ might suggest that the alphabet is a
finite field and most coding theory applications assume this, but we will not make use of a field structure.
In particular, $q$ is not necessarily a power of a prime.

In this paper we shall find it convenient to use the ``inner product"
\[ \langle x,y \rangle := 1-\frac{2d(x,y)}{n}\]
instead of the distance $d(x,y)$ (see \cite{Lev}).
The set of all possible distances   and the set of all possible  inner products between points in $\mathbb{H}(n,q)$     will be denoted by  $Z_n$ and $T_n$, respectively; i.e., $Z_n=\{0,1,\ldots, n\}$ and $T_n=\{t_0,t_1,\ldots,t_n\}$ where
$t_i :=-1+\frac{2i}{n}$, $i=0,1,\ldots,n$.  In addition to the discrete sets $T_n\subset [-1,1]$ and $Z_n\subset [0,n]$
we shall use  the complete intervals $0\le z\le n$ and $-1\le t\le 1$ with dependent variables $z$ and $t$ related through  the equation
\begin{equation}\label{zt}
t = 1 - \frac{2z}{n}.
 \end{equation}

We refer to any nonempty set $C \subset \mathbb{H}(n,q)$ as a {\em code}.
For a given potential function
$h:[-1,1) \to (0,+\infty)$, we define the {\sl $h$-energy}   of $C$ by
\begin{equation}
E(n,C;h):=\frac{1}{|C|}\sum_{x, y \in C, x \neq y} h(\langle x,y \rangle),
\end{equation}
where $|C|$ denotes the cardinality of $C$.
Many problems of interest (cf. \cite{AB,ABL,BHS,CZ})  can be formulated as minimizing
   the quantity $E(n,C;h)$  for a suitable $h$  over codes $C$ of fixed   cardinality; that is, to determine
\begin{equation} \label{minEn} {\mathcal E}(n,M;h):=\min \{E(n,C;h):|C|=M\},
\end{equation}
the   minimal $h$-energy of a code $C \subset \mathbb{H}(n,q)$ of cardinality $M$.
While we only need the values of $h$ on the discrete set $T_n$ for computing the $h$-energy, in this paper we shall further assume that $h$ is {\em (strictly) absolutely monotone} on
the interval [-1,1); that is, $h$ and all its derivatives are defined and (positive) nonnegative on this interval.  We remark that $F(z)=h(t)$, where $z$ is given by \eqref{zt}, is   {\em completely monotone} on $(0,n]$ (that is,  $(-1)^k F^{(k)}(z)\ge 0$ for all $z\in (0,n]$) if and only if $h$ is absolutely monotone on $[-1,1]$.

Important examples of such potentials include
Riesz $\alpha$-potentials $h(t):=( n(1-t)/2)^{-\alpha}=z^{-\alpha} $ for $\alpha>0$ and also  exponential potentials $h(t)=\exp(\alpha t)$.  The energy of the latter potentials is used
in estimating the error probability for random codewords \cite{ME}.  Additional applications of energy problems in Hamming spaces are given in \cite{AB,CZ}.
We further remark that energy minimizing codes $C \subset \mathbb{H}(n,q)$ for the Riesz $\alpha$-potentials or $\alpha$-exponential potentials as mentioned above, approach {\em maximal codes}; that is, codes that  maximize the minimum distance
$d(C):=\min\{d(x,y):x,y \in C, x \neq y\}$ as $\alpha\to \infty$.


 In \cite{CZ} Cohn and Zhao  studied   minimal energy problems for  $\mathbb{H}(n,q)$ and found
codes $C$ that are {\em universally optimal} in the sense of \cite{CK}; i.e., $E(n,C;h)= {\mathcal E}(n,|C|;h)$ for a large class of potential functions $h$.
Specifically, they consider $h$ of the form
\begin{equation}\label{hf}
h(\langle x, y\rangle)=F(d(x,y))
\end{equation} where $F:Z_n\setminus\{0\}\to \R$ is {\em  completely monotone} in a discrete sense; namely \linebreak $(-1)^k\Delta^k F(j)\ge 0$ for $k\ge 0$ and $j=1, \ldots, n-k$, where
$\Delta$ is the forward difference operator   $\Delta F(j):=F(j+1)-F(j)$.
Note that if a function $F$ is completely monotone on $[0,n]$ in the continuous sense, then its restriction to $Z_n$ will be completely monotone in the discrete sense as $\Delta^k F(i)=F^{(k)} (\xi)$ for some $\xi \in (i,i+k)$. Our setting, while somewhat more restrictive than   in   \cite{CZ},
allows for a unified definition,
proof, and investigation of universal (in sense of Levenshtein) bounds
(Theorems \ref{thm 6}, \ref{thm 7} and \ref{thm4.1}).  Furthermore, such a continuous setting  facilitates an asymptotic analysis  (as $n \to \infty$) of our bounds.

In this paper we obtain universal lower bounds for ${\mathcal E}(n,M;h)$, where the universality is meant in Levenshtein's sense (bounds hold for all dimensions and cardinalities, cf. \cite{Lev}),
as well as in expressions which are common for a large class of potential functions.
Our bounds are attained for many well known good codes
(e.g. Hamming, Golay, MDS, or Nordstrom-Robinson codes) which are universally optimal in the sense
of \cite{CZ} (see also \cite[Section 6.2, Table 6.4]{Lev}, \cite{Lev92,Lev95}, \cite{AB}). Furthermore,
our bounds
depend on the application of a quadrature rule to the potential, where the nodes and weights of the quadrature rule are independent of the potential. It turns out that this quadrature coincides with the one utilized by Levenshtein in \cite{Lev92} (see also \cite{Lev95,Lev}) to determine his bound on maximal codes (see Subsection 2.4).

In Section 2 we collect the main notions and results that are necessary for the derivation and explanation
of our bounds. Section 3 is devoted to general bounds on the energy of codes and designs in $\mathbb{H}(n,q)$ that are derived
in a unified way from identity \eqref{main}. These results are more or less folklore. In Section 4 we state and prove our
main universal lower bounds for codes and designs in Hamming spaces. Although our bounds are optimal in the
sense that they cannot be improved in a certain wide framework, it is still possible to find better
bounds by linear programming. Section 5 describes three ways for finding such improvements -- using the
discrete structure of the inner products (the set $T_n$), using higher
degree polynomials, or using preliminary information on the structure of codes and designs under consideration.
Examples of upper energy bounds are given in Section 6 where we investigate in detail the case of 2-designs in the
binary case $\mathbb{H}(n,2)$. Section 7 is devoted to asymptotic consequences of our lower bounds in a
natural process when the length $n$ and the cardinality $M$ tend to infinity in certain relation.
In Section 8 we provide some examples.

\section{Preliminaries}

\subsection{Krawtchouk polynomials and the linear programming framework}

For fixed $n$ and $q$, the (normalized) Krawtchouk polynomials are
defined by
\begin{equation}\label{Kraw} Q_i^{(n,q)}(t) :=\frac{1}{r_i} K_i^{(n,q)}(z), \end{equation}
where $z=\frac{n(1-t)}{2}$, $r_i=r_i^{(n)}:=(q-1)^i {n \choose i}$, and
\[ K_i^{(n,q)}(z):=\sum_{j=0}^i (-1)^j(q-1)^{i-j} {z \choose j} {n-z \choose i-j}, \ \ i=0,1,\ldots,n, \]
are the (usual) Krawtchouk polynomials corresponding to $\mathbb{H}(n,q)$. The polynomials $K_i^{(n,q)}(z)$
can be   defined by
\begin{equation}\label{KrawInit}
K_0^{(n,q)}(z)=1, \ \ K_1^{(n,q)}(z)=n(q-1)-qz,
\end{equation} and, for $1 \leq i \leq n-1$, the three-term recurrence relation
\begin{equation}\label{Kraw3term}
 (i+1)K_{i+1}^{(n,q)}(z)=[i+(q-1)(n-i)-qz]K_i^{(n,q)}(z)-(q-1)(n-i+1)K_{i-1}^{(n,q)}(z).
\end{equation}
 The measure of orthogonality for the system $\{ Q_i^{(n,q)}(t) \}_{i=0}^n$ is a discrete measure given by
\begin{equation} \label{KrawOrtho} d\mu_n (t) := q^{-n}\sum_{i=0}^n {n \choose i} (q-1)^i \delta_{t_i},\end{equation}
where $\delta_{t_i} $ is the Dirac-delta measure at $t_i \in T_n$. Note that the form
\begin{equation}\label{InnerProd}
\langle f,g \rangle=\int f(t) g(t) d\mu_n (t)
\end{equation}
defines an inner product over the class $\mathcal{P}_n$ of polynomials of degree less than or equal to $n$.

We also need the so-called {\sl adjacent polynomials} as introduced by Levenshtein (cf. \cite[Section 6.2]{Lev}, see also \cite{Lev92,Lev95})
\begin{eqnarray}
\label{adjacent}
Q_i^{(1,0,n,q)}(t) &=& \frac{K_i^{(n-1,q)}(z-1)}{\sum_{j=0}^i {n \choose j} (q-1)^j}, \\
\label{adjacent2}
Q_i^{(1,1,n,q)}(t) &=& \frac{K_i^{(n-2,q)}(z-1)}{\sum_{j=0}^i {n-1 \choose j} (q-1)^j}, \\
\label{adjacent3}
Q_i^{(0,1,n,q)}(t) &=& \frac{K_i^{(n-1,q)}(z)}{{n-1 \choose i} (q-1)^i},
\end{eqnarray}
where $z=n(1-t)/2$. The corresponding measures of orthogonality are, respectively,
\begin{equation}\label{KrawOrtho2} (1-t)d\mu_n (t), \quad (1-t)(1+t)d\mu_n(t), \quad (1+t)d\mu_n (t).\end{equation}

An important role in our analysis is played by the following analog of spherical harmonics.
Let $V_0$ consist of the constant function $1$ and, for $i=1,\ldots, n$, let $V_i$ consist of the $r_i:=(q-1)^i {n \choose i}$ functions
\begin{equation*}
\begin{split}V_i=\{u(x):\mathbb{H}(n,q) &\to \mathbb{C} \mid u(x)=\xi^{\alpha_1 x_{j_1}+\cdots+\alpha_i x_{j_i}},\\
&1 \leq j_1<\cdots<j_i \leq n, \alpha_1,\ldots,\alpha_i \in \{1,\ldots,q-1\}\},
\end{split}
\end{equation*}
where $\xi$ is a (complex) primitive $q$-th root of unity.  We denote and enumerate the functions in $V_i$ by $Y_{ij}$, $j=1,\ldots, r_i$.
It is easy to verify that $V:= \{Y_{ij}: 0\le i\le n,\,  1\le j\le r_i\}$ is an orthonormal system with respect to the
inner product $\langle u,v\rangle =q^{-n} \sum_{x \in \mathbb{H} (n,q)} u(x) \overline{v(x)}$  (see \cite[Theorem 2.1]{Lev95}).

The following addition formula relates the Krawtchouk polynomials \eqref{Kraw} and the orthonormal systems $V_i$, $i=0,\ldots, n$,
\begin{equation}\label{addformula}
Q_i^{(n,q)}(\langle x,y\rangle)=\frac{1}{r_i}\sum_{j=1}^{r_i} Y_{ij}(x)\overline{Y_{ij}(y)}.
\end{equation}

If $f(t) \in \mathbb{R}[t]$ is a real polynomial of degree $m \leq n$, then
$f(t)$ can be uniquely expanded in terms of the Krawtchouk
polynomials as $f(t) = \sum_{i=0}^m f_i Q_i^{(n,q)}(t)$.
For $C \subset \mathbb{H}(n,q)$, the identity (an easy consequence of \eqref{addformula})
\begin{equation}
  \label{main}
  |C|f(1)+\sum_{x,y\in C, x \neq y} f(\langle x,y\rangle)
      = |C|^2f_0 + \sum_{i=1}^m \frac{f_i}{r_i} \sum_{j=1}^{r_i}
        \left ( \sum_{x\in C} Y_{ij}(x) \right )^2
\end{equation}
serves as a key source of estimations by linear programming (see, for example, \cite[Equation (1.7)]{Lev83}, \cite[Equation (1.20)]{Lev92}, \cite[Equation (26)]{Lev95}). The Rao bound (subsection 2.3) and the Levenshtein bound (subsection 2.4) can be obtained after appropriate
sums on both sides of (\ref{main}) are neglected and suitable polynomials (optimal in some sense)
are applied.

\subsection{Designs in $\mathbb{H}(n,q)$ and their energy}\label{designs}

We also need the notion of designs (see \cite{HSS,Lev}) in $\mathbb{H}(n,q)$,
which play an important role in the understanding of
energy problems in Hamming spaces. The designs in Hamming spaces have been well studied
since they are, in a certain sense, an approximation of the whole space $\mathbb{H}(n,q)$.

We first give a combinatorial definition.

\begin{defn}
Let $\tau$ and $\lambda$ be positive integers.
A $\tau$-design $C \subset \mathbb{H}(n,q)$ of strength $\tau$ and index $\lambda$
is a code $C \subset \mathbb{H}(n,q)$ of cardinality $|C|=M= \lambda q^\tau$
such that the $M\times n$ matrix obtained from the
codewords of $C$ as rows has the following property: every $M \times \tau$ submatrix
contains every element of $\mathbb{H}(\tau,q)$  exactly $\lambda=\frac{M}{q^{\tau}}$ times as rows.
\end{defn}

An equivalent definition (see \cite[Corollary 2.2]{Lev95}) asserts that $C \subset \mathbb{H}(n,q)$ is a $\tau$-design if
and only if
\begin{equation}\label{eqDef}\sum_{x\in C} Y_{ij}(x)=0, \quad\text{ for   $1 \leq i \leq \tau$ and $1 \leq j \leq r_i$.}
\end{equation}
 This reduces the right-hand side of \eqref{main}
to $f_0|C|^2$ for polynomials of degree at most $\tau$ and thus suggests that such polynomials could
be very useful in derivation and investigation of linear programming bounds for designs.

The characterization of codes by their strength as designs started with Delsarte \cite{Del},
where $\tau+1=d^\prime$ is the dual distance of the code $C$ (see also \cite{DL,Lev95,Lev}).


In this paper we obtain bounds for the minimum and maximum possible potential energies
of designs in $\mathbb{H}(n,q)$. We denote by  ${\mathcal L}(n,M,\tau;h)$ and  ${\mathcal U}(n,M,\tau;h)$   the minimum and maximum, respectively,  of the $h$-energy of $M$-point $\tau$-designs in $\mathbb{H}(n,q)$; that is,
\begin{equation}\label{LUdef}
\begin{split} {\mathcal L}(n,M,\tau;h) := \min \{E(n,C;h): |C| = M, C \subset \mathbb{H}(n,q)  \mbox{ is a $\tau$-design}\},   \\
  {\mathcal U}(n,M,\tau;h) := \max \{E(n,C;h): |C| = M, C \subset \mathbb{H}(n,q) \mbox{ is a $\tau$-design}\}.
  \end{split}
  \end{equation}
In the event  there is no $\tau$-design with cardinality $M$ in $\mathbb{H}(n,q)$, we set  ${\mathcal L}(n,M,\tau;h)  =\infty$ and ${\mathcal U}(n,M,\tau;h)  =-\infty$ as is standard for the inf and sup of the empty set.

All the quantities ${\mathcal E}(n,M;h)$, ${\mathcal L}(n,M,\tau;h)$,  and ${\mathcal U}(n,M,\tau;h)$,
 will be estimated by polynomials techniques
(linear programming method) by using suitable
polynomials in \eqref{main}.

\subsection{Rao bound}

For fixed strength $\tau$ and dimension $n$ denote
\[ B(n,\tau) = \min \{|C| : \mbox{$\exists$ $\tau$-design $C \subset \mathbb{H}(n,q)$}\}. \]
The classical universal lower bound on $B(n,\tau)$ is due to Rao \cite{Rao} (see also \cite{HSS,DL,Lev})
\begin{equation}
\label{R-bound}
B(n,\tau) \geq R(n,\tau) := \left\{
  \begin{array}{ll}
    \ds q\sum_{i=0}^{k-1} {n-1 \choose i}(q-1)^i,       & \mbox{if $\tau = 2k-1$,} \\[10pt]
    \ds \sum_{i=0}^k {n \choose i}(q-1)^i, & \mbox{if $\tau = 2k$}.
  \end{array} \right.
\end{equation}

The bound \eqref{R-bound} can be obtained by linear programming using in \eqref{main} the following polynomials of degree $\tau$:
\[ f^{(\tau)}(t)=(t+1)^\varepsilon \left(\sum_{i=0}^k r_i^{(0,1)} Q_i^{(0,1,n,q)}(t) \right)^2, \]
where $\tau=2k-1+\varepsilon$, $\varepsilon \in \{0,1\}$, $r_i^{(0,1)}=(q-1)^i {n-1 \choose i}$ for
$i=0,1,\ldots,n-1$.

We use the Rao bound to indicate which parameters must be chosen in order to obtain
universal lower bounds on ${\mathcal E}(n,M;h)$ and ${\mathcal L}(n,M,\tau;h)$  and upper bounds for ${\mathcal U}(n,M,\tau;h)$. More precisely, for given
length $n$ and cardinality $M$, we find the unique
\[ \tau:=\tau(n,M) \mbox{ \ such that \ } M \in \left(R(n,\tau),R(n,\tau+1)\right]. \]
Then all other necessary parameters come with $n$, $M$ and $\tau$ as shown in Subsection 2.5.

\subsection{Levenshtein bound}

Let
$$A_q(n,s):=\max\{|C| \colon C \subset \mathbb{H}(n,q), \langle x,y \rangle \leq s, \,   x\neq y \in C\}$$
denote the maximal possible cardinality of a code in $\mathbb{H}(n,q)$ of prescribed maximal
inner product $s$. We remark that in Coding Theory this quantity is usually denoted by $A_q(n,d)$, where $s=1-\frac{2d}{n}$,
so we have replaced the condition $d(x,y) \geq d$ by $\langle x,y \rangle \leq s$.

For $a,b \in \{0,1\}$ and $i \in \{ 1,2,\ldots,n-a-b\}$, denote by $t_i^{a,b}$ the greatest zero of the adjacent
polynomial $Q_i^{(a,b,n,q)}(t)$ (see \eqref{adjacent}-\eqref{adjacent3}) and also define $t_0^{1,1}=-1$. We have the interlacing properties
$t_{k-1}^{1,1}<t_k^{1,0}<t_k^{1,1}$, see \cite[Lemmas 5.29, 5.30]{Lev}. For a positive integer $\tau$, let
$\mathcal{I}_\tau$ denote the interval
\begin{eqnarray*}
  \mathcal{I}_\tau :=
\left\{
\begin{array}{ll}
    \left [ t_{k-1}^{1,1},t_k^{1,0} \right ], & \mbox{if } \tau=2k-1, \\[6pt]
    \left [ t_k^{1,0},t_k^{1,1} \right ],      & \mbox{if } \tau=2k. \\
  \end{array}\right.
\end{eqnarray*}
Then the intervals $\mathcal{I}_\tau$ are well defined and
partition $\mathcal{I}=[-1,1)$ into subintervals with non-overlapping interiors.

For every $\tau$ and $s \in \mathcal{I}_\tau$, Levenshtein obtained
the `linear programming' bound (see \cite[Equation (6.45) and (6.46)]{Lev})
\begin{equation}
\label{L_bnd}
 A_q(n,s) \leq
\left\{
\begin{array}{ll}
    L_{2k-1}(n,s) =
         \left(1 - \frac{Q_{k-1}^{(1,0,n,q)}(s)}{Q_k^{(n,q)}(s)}\right)\sum\limits_{j=0}^{k-1} {n \choose j}(q-1)^{j},
          &\mbox{if }\tau={2k-1},\\[10pt]
    L_{2k}(n,s) =
        q\left(1 - \frac{Q_{k-1}^{(1,1,n,q)}(s)}{Q_k^{(0,1,n,q)}(s)}\right)\sum\limits_{j=0}^{k-1} {n-1 \choose j}(q-1)^{j},
        &\mbox{if } \tau={2k}.\cr
   \end{array}\right.
\end{equation}
using \eqref{main} with
the following polynomials of degree $\tau=2k-1+\varepsilon$, $\varepsilon \in \{0,1\}$,
\begin{equation}
\label{lev_poly}
f_\tau^{(n,s)}(t)=(t-s)(t+1)^\varepsilon \left(\sum_{i=0}^{k-1} r_i^{(1,\varepsilon)}
Q_i^{(1,\varepsilon,n,q)}(t) Q_i^{(1,\varepsilon,n,q)}(s)\right)^2
\end{equation}
(see \cite[Equations (5.81) and (5.82)]{Lev}), where
\[ r_i^{(1,\varepsilon)}=\left(\sum_{j=0}^i {n-\varepsilon \choose j} (q-1)^j \right)^2/(q-1)^i {n-1-\varepsilon \choose i}. \]

An important connection between the Rao (\ref{R-bound}) and the
Levenshtein  (\ref{L_bnd}) bounds is given by the equalities
\begin{equation}\label{L-DGS1}
\begin{split}
L_{2k-2}(n,t_{k-1}^{1,1})&=
L_{2k-1}(n,t_{k-1}^{1,1}) = R(n,2k-1),\\
  L_{2k-1}(n,t_k^{1,0})&=
L_{2k}(n,t_k^{1,0}) = R(n,2k)
\end{split}
\end{equation}
at the ends of the intervals $\mathcal{I}_\tau$. The relations \eqref{L-DGS1} explain and justify our
connection between the cardinality $M$ and the strength (degree) $\tau=\tau(n,M)$.

\subsection{Useful quadrature}
\label{quadSec}
Levenshtein \cite{Lev92} proves (see also \cite[Section 5, Theorem 5.39]{Lev} and \cite{Lev95}) that
for every fixed (cardinality) $M > R(n,2k-1)$ there exist
uniquely determined real numbers $-1 < \alpha_0 < \alpha_1 <
\cdots <\alpha_{k-1} < 1$ and positive numbers $\rho_0,\rho_1,\ldots,\rho_{k-1}$,
such that the equality
\begin{equation}
\label{quad_f0_alfi}
f_0= \frac{f(1)}{M}+ \sum_{i=0}^{k-1} \rho_i f(\alpha_i)
\end{equation}
holds for every real polynomial $f(t)$ of degree at most $2k-1$.

The numbers $\alpha_i$, $i=0,1,\ldots,k-1$, are the roots of the
equation
\begin{equation}
\label{alfi}
P_k(t)P_{k-1}(\alpha_{k-1}) - P_k(\alpha_{k-1})P_{k-1}(t)=0,
\end{equation}
where $P_i(t)=Q_i^{(1,0,n,q)}(t)$. In fact, $\alpha_i$, $i=0,1,\ldots,k-1$,
are also the roots of the polynomial $f_{2k-1}^{(n,s)}(t)$ from \eqref{lev_poly}
(see \cite{Lev92,Lev}). In our approach, it is convenient to
find $\alpha_{k-1}=s$ from the equation $M=L_{2k-1}(n,s)$ and to solve then \eqref{alfi}.

Similarly, for every fixed (cardinality) $M > R(n,2k)$ there exist
uniquely determined real numbers $-1=\beta_0 < \beta_1 < \cdots <\beta_k < 1$
and positive numbers $\gamma_0,\gamma_1,\ldots,\gamma_k$,
such that the equality
\begin{equation}
\label{quad_f0_beti}
f_0= \frac{f(1)}{M}+ \sum_{i=0}^{k} \gamma_i f(\beta_i)
\end{equation}
holds for every real polynomial $f(t)$ of degree at most $2k$.
The numbers $\beta_i$, $i=1,\ldots,k$, are the roots of the
equation
\begin{equation}
\label{beti}
P_k(t)P_{k-1}(\beta_k) - P_k(\beta_k)P_{k-1}(t)=0,
\end{equation}
where $P_i(t)=Q_i^{(1,1,n,q)}(t)$ and also roots of the polynomial $f_{2k}^{(n,s)}(t)$
from \eqref{lev_poly} (see \cite{Lev92,Lev}). Similarly to the odd case, $\beta_k=s$ can be found from the equation $M=L_{2k}(n,s)$
and then \eqref{beti} can be solved.

As mentioned in the end of subsection 2.3 we always take care where the cardinality $M$ is located with respect to the
Rao bound. We actually associate $M$ with the corresponding numbers:
\[ \alpha_0,\alpha_1,\ldots,\alpha_{k-1},\rho_0,\rho_1,\ldots,\rho_{k-1} \mbox{ when } M=L_{2k-1}(n,s) \in (R(n,2k-1),R(n,2k)] \]
or, analogously, with the corresponding
\[ \beta_0,\beta_1,\ldots,\beta_{k},\gamma_0,\gamma_1,\ldots,\gamma_{k} \mbox{ when } M=L_{2k}(n,s) \in (R(n,2k),R(n,2k+1)]. \]
Then the quadratures \eqref{quad_f0_alfi} and \eqref{quad_f0_beti} can be applied for deriving and calculation of our bounds.

We also use the kernels (see (2.69) in \cite[Section 2]{Lev}; also Section 5 in \cite{Lev})
\begin{equation}
T_k(u,v)=\sum_{i=0}^k r_i Q_i^{(n,q)}(u)Q_i^{(n,q)}(v)=c \cdot \frac{Q_{k+1}^{(n,q)}(u)Q_k^{(n,q)}(v)-Q_{k+1}^{(n,q)}(v)Q_k^{(n,q)}(u)}{u-v}
\end{equation}
($c$ is a positive constant, $u \neq v$, this is in fact the Christoffel-Darboux formula). Note that the $(1,0)$ and $(1,1)$
analogs of $T_k(u,v)$ define the Levenshtein polynomials -- see \eqref{lev_poly}.

\subsection{Bounds for the extreme inner products of designs in $\mathbb{H}(n,q)$}

Setting $s(C):=\max \{\langle x,y \rangle : x,y \in C, x \neq y\}$ and
$\ell(C):=\min \{\langle x,y \rangle : x,y \in C, x \neq y\}$, we define the quantities
\begin{equation}\label{slDef}
\begin{split}
 s(n,M,\tau)&:=\max \{s(C):C \subset \mathbb{H}(n,q) \mbox{ is a $\tau$-design}, |C|=M\},  \\
 \ell(n,M,\tau)&:=\min \{\ell(C):C \subset \mathbb{H}(n,q) \mbox{ is a $\tau$-design}, |C|=M\},
 \end{split}
 \end{equation}
  with the previously adopted convention regarding  the case when no $M$-point $\tau$-design exists.

 The following equivalent  definition of designs in $\mathbb{H}(n,q)$ is useful in obtaining bounds for the above quantities;
 see Subsection 5.3 and Section 6.

\begin{defn}
\label{def_des_alg}
A $\tau$-design $C \subset \mathbb{H}(n,q)$ is a code such that
the equality
\begin{equation}
\label{defin_f.3} \sum_{y \in C} f(\langle x,y \rangle ) = f_0|C|
\end{equation}
holds for any point $x \in \mathbb{H}(n,q)$ and any real
polynomial $f(t)=\sum_{i=0}^r f_iQ_i^{(n,q)}(t)$ of degree $r \leq \tau$.
\end{defn}

\section{General linear programming bounds for ${\mathcal E}(n,M;h)$,
${\mathcal L}(n,M,\tau;h)$,  and ${\mathcal U}(n,M,\tau;h)$}

In this section we state some immediate consequences of \eqref{main} and \eqref{eqDef}.

We begin with the  following analog for Hamming spaces of Yudin's well-known lower bound  for the energy of spherical codes \cite{Yud}; see also \cite[Proposition 5]{CZ}.

\begin{theorem}\label{thm 1}
Let $n$ be a positive integer and $h$ be an absolutely monotone function on $[-1,1)$.  If $f$ is a real polynomial satisfying

{\rm (A1)} $f(t) \leq h(t)$ for every $t \in T_n$, and

{\rm (A2)} the coefficients in the Krawtchouk expansion $f(t) = \sum_{i=0}^{n} f_i Q_i^{(n,q)}(t)$
satisfy $f_i \geq 0$ for every $i \geq 1$,
then  \begin{equation}{\mathcal E}(n,M;h) \geq f_0M-f(1) \qquad \text{for every $M\ge 2$.}\end{equation}
\end{theorem}

The $\tau$-design  property allows relaxation of the conditions on the coefficients in the Krawtchouk
expansion.   In the following two theorems the reader should note that the conclusions hold trivially when there are no   $M$-point $\tau$-designs
in $\mathbb{H}(n,q)$;  see remark following \eqref{LUdef} regarding the max and min for the empty set.

\begin{theorem}
\label{thm 2}
Let $n$  and $h$ be as in Theorem~\ref{thm 1}.  If $\tau$ is a  positive integer    and $f$ is a real polynomial
 that satisfies {\rm (A1)} and

{\rm (A2$^\prime$)} the coefficients in the Krawtchouk expansion $f(t) = \sum_{i=0}^{n} f_i Q_i^{(n,q)}(t)$
satisfy $f_i \geq 0$ for every $i \geq \tau+1$,
then
\begin{equation}
{\mathcal L}(n,M,\tau;h) \geq f_0M-f(1) .
\end{equation}
\end{theorem}

Denote by $A_{n,M;h}$ (respectively $A_{n,M,\tau;h}$) the set of polynomials that satisfy
the conditions (A1) and (A2) (respectively (A1) and (A2$^\prime$)).

One similarly obtains general upper bounds for ${\mathcal U}(n,M,\tau;h)$.

\begin{theorem}\label{thm 4}
Let $n$   and $h$ be as in Theorem~\ref{thm 1}.  If $\tau$ and $M$ are positive integers and $g$ is a real polynomial satisfying

{\rm (B1)} $g(t) \geq h(t)$ for every $t \in T_n \cap [\ell(n,M,\tau),s(n,M,\tau)]$ and

{\rm (B2)} the coefficients in the Krawtchouk expansion $g(t) = \sum_{i=0}^{n} g_i Q_i^{(n,q)}(t)$
satisfy $g_i \leq 0$ for $i \geq \tau+1$,
then
\begin{equation}
{\mathcal U}(n,M,\tau;h) \leq g_0M-g(1).\end{equation}
\end{theorem}

Denote by $B_{n,M,\tau;h}$ the set of polynomials satisfying the conditions
(B1) and (B2).

Theorems \ref{thm 1}--\ref{thm 4} suggest the following optimization problems.

\begin{problem}\label{prb 1}
Find polynomial(s) $f \in A_{n,M;h}$ ($f \in A_{n,M,\tau;h}$ respectively)
that give the maximum value of $f_0M-f(1)=f_0(M-1)-(f_1+\cdots+f_{n}).$
\end{problem}

\begin{problem}\label{prb 2}
Find polynomial(s) $g \in B_{n,M,\tau;h}$
that give the minimum value of $g_0M-g(1)=g_0(M-1)-(g_1+\cdots+g_{n}).$
\end{problem}

Another general problem asks for finding universally optimal codes \cite{ABL,CZ}.
Such codes attain the minimum possible energy with respect to all absolute monotone potential
functions simultaneously.

\section{Universal lower bounds for ${\mathcal L}(n,M,\tau;h)$ and ${\mathcal E}(n,M;h)$}
The following theorem is an immediate corollary of Theorem~\ref{thm 7}.  We
 include the proof of this theorem since it is simpler  and will be referenced in the proof of Theorem~\ref{thm 7}.

\begin{theorem}
\label{thm 6}
If  $n$ and $\tau$ are positive integers and $h$ is absolutely monotone on $[-1,1)$, then
\begin{equation}
\label{bound_odd_designs}
{\mathcal L}(n,M,\tau;h) \geq \begin{cases}
\ds M\sum_{i=0}^{k-1} \rho_i h(\alpha_i),& \tau=2k-1, \ \forall M \in \left(R(n,2k-1),R(n,2k)\right], \\[10pt]
\ds  M\sum_{i=0}^{k} \gamma_i h(\beta_i),& \tau=2k, \ \forall M \in \left(R(n,2k),R(n,2k+1)\right].
\end{cases}
\end{equation}
Moreover, if ${\mathcal L}(n,M,\tau;h)<\infty$, the bounds \eqref{bound_odd_designs} cannot be improved by utilizing polynomials $f$ of degree at most  $\tau$
satisfying $f(t) \leq h(t)$ for every $t \in [-1,1)$; that is, the right-hand side of \eqref{bound_odd_designs} equals
the maximum value of   $f_0 M-f(1)$ over all such polynomials.

\end{theorem}

\begin{proof}
Let $\tau=2k-1$ and the polynomial $f(t)$ be the Hermite interpolant of $h(t)$ at the points $\alpha_0,\alpha_1,\ldots,\alpha_{k-1}$,
i.e. $f(\alpha_i)=h(\alpha_i)$ and $f^\prime(\alpha_i)=h^\prime(\alpha_i)$ for every $i=0,1,\ldots,k-1$.
Then $\deg(f) \leq 2k-1$ and the condition (A2)$^\prime$ is trivially satisfied. Furthermore, it follows from Rolle's Theorem
that $f(t) \leq h(t)$ for every $t \in [-1,1)$. Therefore (A1) is also satisfied and $f(t) \in A_{n,M,2k-1;h}$.
We calculate the bound by using the quadrature formula \eqref{quad_f0_alfi}:
\[ f_0 = \frac{f(1)}{M}+ \sum_{i=0}^{k-1} \rho_i f(\alpha_i) \iff f_0M-f(1)=M\sum_{i=0}^{k-1} \rho_i f(\alpha_i) \]
and the last equality implies $f_0M-f(1) = M \sum_{i=0}^{k-1} \rho_i h(\alpha_i)$
since $f(\alpha_i) =h(\alpha_i)$ from the interpolation.

Furthermore, for any polynomial $F(t)$ of degree at most $2k-1$ satisfying $F(t) \leq h(t)$ for every $t \in [-1,1)$,
we have from the quadrature formula (\ref{quad_f0_alfi}) for $f(t)$ and for $F(t)$
\[ f_0M-f(1)=M\sum_{i=0}^{k-1} \rho_i f(\alpha_i)=M\sum_{i=0}^{k-1} \rho_i h(\alpha_i) \geq M\sum_{i=0}^{k-1} \rho_i F(\alpha_i)=F_0M-F(1),\]
which proves the optimality property of $f(t)$.

The case $\tau=2k$ is similar, with single intersection $f(\beta_0)=h(\beta_0)$
of the graphs of $f(t)$ and $h(t)$ at the point $\beta_0=-1$
and $f(\beta_i)=h(\beta_i)$ and $f^\prime(\beta_i)=h^\prime(\beta_i)$ for every $i=1,\ldots,k$.
Now the degree of $f(t)$ is at most $2k$ and again (A2)$^\prime$ is trivially satisfied
and (A1) follows from Rolle's Theorem.
\end{proof}

The Hermite interpolant of $h(t)$ at the $\{ \alpha_i \}$ nodes can be also utilized for
obtaining the same bounds for $\mathcal{E}(n,M;h)$. The proof of the positive-definiteness of
these Hermite interpolants (i.e. the condition (A2)) follows the framework of \cite[Section 3]{CK},
applied to discrete orthogonal polynomials.

\begin{theorem}\label{thm 7}
If  $n$ is a positive integer and $h$ is absolutely monotone on $[-1,1)$, then
\begin{equation}
\label{bound_codes_odd}
{\mathcal E}(n,M;h) \geq
\begin{cases}
\ds M\sum_{i=0}^{k-1} \rho_i h(\alpha_i),& \forall M \in \left(R(n,2k-1),R(n,2k)\right], \\[10pt]
\ds  M\sum_{i=0}^{k} \gamma_i h(\beta_i),& \forall M \in \left(R(n,2k),R(n,2k+1)\right].
\end{cases}
\end{equation}
Moreover, the bounds \eqref{bound_codes_odd} cannot be improved by utilizing polynomials $f$ of degree at most  $\tau$
satisfying $f(t) \leq h(t)$ for every $t \in [-1,1)$.
\end{theorem}

\begin{proof}
The polynomial from Theorem \ref{thm 6} serves as the solution of the linear program given in Theorem \ref{thm 1}. We already have established (A1). Since we do not have the design property, we need to verify that (A2) is satisfied. We shall do this by adapting the approach in \cite[Sections 3 and 5]{CK}, where for the two cases in the right-hand side of \eqref{bound_codes_odd} we shall consider the discrete measures $(1-t)d\mu_n (t)$ and $(1-t^2)d\mu_n (t)$ and the associated orthogonal polynomials $\{Q_i^{(1,0,n,q)}(t) \}$ and $\{Q_i^{(1,1,n,q)}(t) \}$  respectively.

The results in \cite[Section 3]{CK} (and in particular Theorem 3.1 there) are proved for a Borel measure $\mu$ such that
\[ \int p(t)^2\, d\mu(t) >0 \]
for all polynomials $p$ that are not identically zero and for the associated orthogonal polynomials. However, a careful inspection of the proofs in that section reveals that the results remain true for the discrete measures defined in \eqref{KrawOrtho} and \eqref{KrawOrtho2} and the associated orthogonal  polynomials of degree up to $n$. This is consequence from the fact that for such polynomials the bilinear form \eqref{InnerProd} indeed defines an inner product.

Since the conductivity property for Hermite interpolants, as discussed in \cite[Section 5]{CK}, is independent
of the measure of orthogonality, what remains to prove is the positivity of the constant $-Q_k^{(1,0,n,q)}(\alpha_{k-1})/Q_{k-1}^{(1,0,n,q)}(\alpha_{k-1}) $, respectively $-Q_k^{(1,1,n,q)}(\beta_{k})/Q_{k-1}^{(1,1,n,q)}(\beta_{k}) $.

This follows from the normalization $Q_i^{(1,0,n,q)}(1)=1$, respectively $Q_i^{(1,1,n,q)}(1)=1$,
and the interlacing property of the zeros of the orthogonal polynomials $Q_k^{(1,0,n,q)}(t)$ and
$Q_{k-1}^{(1,0,n,q)}(t)$, respectively $Q_k^{(1,1,n,q)}(t)$ and $Q_{k-1}^{(1,1,n,q)}(t)$ (see \cite[Theorem 3.3.2]{Sze}).

The optimality of the polynomial from Theorem \ref{thm 6} was already derived in
the proof of the previous theorem. This completes the proof.
\end{proof}

\begin{remark}
We note that the roots of the equations \eqref{alfi} and \eqref{beti} interlace with the zeros of
$\{Q_k^{(1,0,n,q)}(t)\}$ and $\{Q_k^{(1,1,n,q)}(t)\}$ respectively (see \cite[Theorem 3.3.4]{Sze}),
which in turn interlace the Krawtchouk polynomials $\{Q_k^{(n,q)}(t)\}$ defined
in \eqref{Kraw} (see \cite[Lemmas 5.29, 5.30]{Lev}).
Therefore, the asymptotic distribution of the quadrature nodes
$\{ \alpha_i \}$, respectively $\{ \beta_i \}$, as $k/n \to const$ when $k,n \to \infty$,
is governed by a constrained energy problem as studied in \cite{DS1} and \cite{DS2}.
We shall investigate this behavior in details in a future work.
\end{remark}

It is clear that all maximal codes which attain the Levenshtein bounds $L_\tau(n,s)$
have the necessary strength and the suitable inner products and therefore achieve
our bounds \eqref{bound_odd_designs} and \eqref{bound_codes_odd} as well.
In \cite{CZ}, a table of universally optimal codes in $\mathbb{H}(n,q)$ is presented.
Some of these codes attain \eqref{L_bnd} and our bounds.
Clearly, all codes on \eqref{bound_odd_designs} and \eqref{bound_codes_odd} are universally optimal.


In the end of this section we remark that the bounds (\ref{bound_odd_designs}) and \eqref{bound_codes_odd}
are discrete analogs of bounds on the potential energy of
spherical codes and designs recently obtained by the authors \cite{BDHSS1,BDHSS2}.

\section{On optimality of the universal lower bounds}

In this section we consider three different approaches for finding better than the universal bounds.

\subsection{Using the discrete nature of the inner products}

Utilizing Cohn-Zhao's \cite{CZ} approach to finding good polynomials we could sometimes improve
the universal bounds by using Lagrange instead of Hermite interpolation
in a set of inner products which correspond to what is called pair covering in \cite{CZ}.
In fact, the nodes $(\alpha_i)$ for the odd case and $(\beta_i)$ for the even case
show which pairs must be covered.

Indeed, if $\alpha_i \in [t_j,t_{j+1})$, then we can replace the Hermite's touching of the
graphs of $f$ and $h$ at the point $\alpha_i$ by intersection in the points $t_j$ and $t_{j+1}$.
This means that the graph of $f$ goes above the graph of
$h$ in the interval $(t_j,t_{j+1})$ and, in particular, $f(\alpha_i) \geq h(\alpha_i)$.
In other words, the idea is to replace the conditions
$f(\alpha_i)=h(\alpha_i)$ and $f^\prime(\alpha_i)=h^\prime(\alpha_i)$ by
$f(t_j)=h(t_j)$ and $f(t_{j+1})=h(t_{j+1})$. Of course, this sometimes needs adjustments as explained
below.

We consider in more detail the odd case $\tau=2k-1$.
Observe that the nodes $\alpha_i$, $i=0,1,\dots,k-2$, as zeros of orthogonal polynomials
w.r.t. the discrete measure $(t-\alpha_{k-1})(1-t) d\mu_n (t)$
are separated by the mass points of $\mu_n$, that is the inner products $T_n$.


For every $i=0,1,\ldots,k-1$, define the pair $(t_{j(i)},t_{j(i)+1})$ of neighboring
elements of $T_n$ by  $t_{j(i)} \leq \alpha_i < t_{j(i)+1}$; i.e.,
$\alpha_i \in [t_{j(i)},t_{j(i)+1})$. If
$t_{j(i)+1}<t_{j(i+1)}$ for every $i \in \{0,1\ldots,k-2\}$,
then the pairs $(t_{j(i)},t_{j(i)+1})$, $i=0,1,\ldots,k-1$, are disjoint. Then
the Lagrange interpolant $f(t)$ of $h(t)$ in the points
\[ t_{j(0)},t_{j(0)+1},t_{j(1)},t_{j(1)+1},\ldots,t_{j(k-1)},t_{j(k-1)+1} \]
has degree at most $2k-1$ and satisfies the property (A1).

If $t_{j(i)+1}=t_{j(i+1)}$ for some $i$ (that is the right end of the interval of $\alpha_i$
coincides with the left end of the interval of $\alpha_{i+1}$) then we apply the Hermite
requirements $f(t_{j(i+1)})=h(t_{j(i+1)})$ and $f^\prime(t_{j(i+1)})=h^\prime(t_{j(i+1)})$.
Now the graph of $f(t)$ touches from above the graph of $h(t)$ at the coincidence point $t_{j(i+1)}$
and (A1) is again satisfied.


Summarizing, we see that this construction always implies that the condition (A1) is satisfied.
The condition (A2$^\prime$) is trivially satisfied and therefore $f(t) \in A_{n,M,2k-1;h}$.
Furthermore, our construction implies, as mentioned above, that $f(\alpha_i) \geq h(\alpha_i)$ for every $i=0,1,\ldots,k-1$.
Hence we obtain
\[ \mathcal{L}(n,M,2k-1;h) \geq f_0M-f(1) \geq M \sum_{i=0}^{k-1} \rho_i h(\alpha_i) \]
and the bound of $f(t)$ is at least as good as the universal bounds \eqref{bound_odd_designs}.
Clearly, the improvement is strict if and only if $\alpha_i \not\in T_n$ for at least one $i$.
We conjecture (see Proposition 21 in \cite{CZ}) that the condition (A2) is also always satisfied and \eqref{bound_codes_odd}
can be correspondingly improved.

The even case is dealt analogously.

\subsection{Test functions and improvements by higher degree polynomials}

Let $n$ and $M$ be fixed and $\tau=\tau(n,M)$ be as explained in the end of subsection 2.3. As above,
the equation
$L_\tau(n,s)=M$, $s=\alpha_{k-1}$ or $\beta_k$, defines all necessary parameters as in subsection 2.5. Let
$j \geq \tau+1$ be a positive integer. We consider the following {\em test-functions} in $n$ and $s$:
\begin{equation}
\label{test-functions}
P_j(n,s):=\begin{cases}
\ds \frac{1}{M}+\sum_{i=0}^{k-1} \rho_i Q_j^{(n,q)}(\alpha_i)
                 & \text{for \ $s \in {\mathcal I}_{2k-1}$} \\[10pt]
\ds               \frac{1}{M}+\sum_{i=0}^{k} \gamma_i Q_j^{(n,q)}(\beta_i)
                 & \text{for \ $s \in {\mathcal I}_{2k}$}.
\end{cases}
\end{equation}
The test functions (\ref{test-functions}) were introduced in 1998 in the context of maximal codes in polynomial metric spaces
(which include $\mathbb{H}(n,q)$) by Boyvalenkov and Danev \cite{BD}, where the binary case $\mathbb{H}(n,2)$ was considered in detail.

The next theorem shows that the functions $P_j(n,s)$ give necessary and sufficient conditions for existence
of improving polynomials of higher degrees satisfying the condition $f(t) \leq h(t)$ in $[-1,1]$.


\medskip

\begin{theorem} \label{thm4.1} Let $h$ be strictly absolutely monotone function.
The bounds (\ref{bound_odd_designs}) or (\ref{bound_codes_odd}) can be improved by a
polynomial of degree at least $\tau+1$ from $A_{n,M,\tau;h}$ or $A_{n,M,h}$, satisfying $f(t) \leq h(t)$ in $[-1,1]$  if and only if
$P_j(n,s) < 0$ for some $j \geq \tau+1$. Furthermore, if $P_j(n,s)<0$ for
some $j \geq \tau+1$, then (\ref{bound_odd_designs}) or (\ref{bound_codes_odd})  can be improved by a polynomial
from $A_{n,M,\tau;h}$ or $A_{n,M,h}$, satisfying $f(t) \leq h(t)$ in $[-1,1]$ of degree exactly $j$.
\end{theorem}

\begin{proof} We give a proof for $\tau=2k-1$.

(Necessity)   Suppose $f(t) \in A_{n,M,\tau;h}$ or $A_{n,M,h}$ and $f(t) \leq h(t)$ in $[-1,1]$. Then
\begin{equation}
\label{n1}
f(t)= g(t)+\sum_{j\geq \tau+1}  f_j Q_j^{(n,q)}(t),\nonumber
\end{equation}
for some $g$ of degree at most $\tau$ and $f_j\ge 0$, for $j\geq \tau+1$.
Note that $f_0=g_0$. Furthermore, using \eqref{quad_f0_alfi} for $g(t)$, we obtain
\begin{eqnarray*}
Mf_0- f(1) &=& Mg_0 - f(1) = g(1)+M\sum_{i=0}^{k-1} \rho_i g(\alpha_i) -\left(g(1)+\sum_{j\geq \tau+1} f_j \right) \\
    &=& M\sum_{i=0}^{k-1} \rho_i \left(f(\alpha_i)-\sum_{j\geq \tau+1} f_j Q_j^{(n,q)}(\alpha_i)\right)-\left(\sum_{j\geq \tau+1}
               f_j \right) \\
    &=& M\sum_{i=0}^{k-1} \rho_i  f(\alpha_i)-
               M\sum_{j\geq \tau+1}  f_j \left(\frac{1}{M}+ \sum_{i=0}^{k-1}\rho_i Q_j^{(n,q)}(\alpha_i) \right)\\
    &=& M\sum_{i=0}^{k-1} \rho_i  f(\alpha_i)-M\sum_{j\geq \tau+1} f_jP_j(n,s)
               \le M\sum_{i=0}^{k-1} \rho_i h(\alpha_i),
\end{eqnarray*}
where, for the last inequality, we used $f(t) \in A_{n,M,\tau;h}$ or $A_{n,M,h}$ and $P_j(n,s) \geq 0$.

\smallskip

(Sufficiency) Conversely, assume that $h$ is strictly absolutely monotone and suppose
that $P_j(n,s) <0$ for some $j \geq 2k$.

We shall improve the bound (\ref{bound_odd_designs}) or (\ref{bound_codes_odd}) by using the polynomial
\[ f(t)=\epsilon Q_j^{(n,q)}(t)+g(t), \]
where $\epsilon >0$ and $g(t)$ of degree at most $2k-1$ will be properly chosen.

Denote $\tilde{h}(t):= h(t)-\epsilon Q_j^{(n,q)}(t)$ and select $\epsilon$ such that
$\tilde{h}(t)^{(i)}(t) \geq 0$ on $[-1,1]$ for all $i=0,1,\dots,j$.
For this choice of $\epsilon$ the function $\tilde{h}(t)$ is absolutely monotone.
The polynomial $g(t)$  is chosen then to be the Hermite interpolant of $\tilde{h}$ at the nodes $\{ \alpha_i\}$, i.e.
\[ g(\alpha_i)=\tilde{h}(\alpha_i), \ \ g^\prime(\alpha_i)=\tilde{h}^\prime(\alpha_i), \ i=0,1,\ldots,k-1. \]
Then $g \in A_{n,M,\tau;\tilde{h}}$ implying that $f \in A_{n,M,\tau;h}$ and, since $\tilde{h}(t)$ is an absolutely monotone function,
we can infer as in Theorem \ref{thm 7} that $g \in A_{n,M;\tilde{h}}$, implying that $f\in A_{n,M;h}$.

Let $g(t)=\sum_{\ell=0}^{2k-1} g_\ell Q_\ell^{(n,q)}(t)$. Note that $f_0=g_0$ and $f(1)=g(1)+\epsilon$.
We next prove that the bound given by $f(t)$ is better than the odd branch of \eqref{bound_odd_designs} or
\eqref{bound_codes_odd}. To this end, we
multiply by $\rho_i$ and sum up the first interpolation equalities:
\[ \sum_{i=0}^{k-1} \rho_i g(\alpha_i)= \sum_{i=0}^{k-1} \rho_i h(\alpha_i)-\epsilon \sum_{i=0}^{k-1} \rho_i Q_j^{(n,q)}(\alpha_i). \]
Since $$M\sum_{i=0}^{k-1} \rho_i g(\alpha_i)=Mg_0-g(1)$$ by (\ref{quad_f0_alfi}) and
$$M\sum_{i=1}^{k} \rho_i Q_j^{(n,q)}(\alpha_i)=MP_j(n,s)-1 $$
by the definition of the test functions (\ref{test-functions}), we obtain
\[ Mg_0-g(1)=M\sum_{i=0}^{k-1} \rho_i h(\alpha_i)+\epsilon -M\epsilon P_j(n,s) \]
which is equivalent to
\[ Mf_0-f(1)=M\sum_{i=0}^{k-1} \rho_i h(\alpha_i)-M\epsilon P_j(n,s)>M\sum_{i=0}^{k-1} \rho_i h(\alpha_i), \]
i.e. the polynomial $f(t)$ gives a better bound indeed.
\end{proof}

\medskip

The investigation of the test functions $P_{2k+3}(n,s)$ for $\tau \in \{2k-1,2k\}$
in the binary case from \cite[Section 4.2]{BD} gives the following.

\medskip

\begin{theorem}
\label{test_functions_binary}
Let $q=2$.

a) If $2k+3 \leq n \leq k^2+4k+2$, then the even branch of the bounds \eqref{bound_odd_designs}
and \eqref{bound_codes_odd} can be improved for every $s \in \left(t_k^{1,0},t_k^{1,1}\right)$.

b) If $k \geq 5$ and $2k+3 \leq n \leq \left(k^2+8k+1+\sqrt{(k^2+4k+5)(k^2-4k-3)}\right)/4$,
then the odd branch of the bounds \eqref{bound_odd_designs}
and \eqref{bound_codes_odd} can be improved for every $s \in \left(t_{k-1}^{1,1},t_k^{1,0}\right)$.
\end{theorem}

\subsection{Bounds for codes and designs with inner products in a subinterval of $[-1,1]$}

Some classes of codes or designs are known to have inner products in proper subinterval of $[-1,1]$ and this
implies restrictions on their structure. We proceed with examples for such situations with
$\tau$-designs in the binary case $\mathbb{H}(n,2)$ with even $\tau=2k$ and
of cardinality $M \in (R(n,2k),R(n,2k+1))$.

The next assertion is implicit in \cite[Section 4]{BD_pms} (see also Corollary 5.49 and Remark 5.58 in \cite{Lev}).

\medskip

\begin{lemma}
\label{Lem4.1}
If $C \subset \mathbb{H}(n,2)$ is a $(2k)$-design of cardinality $M \in (R(n,2k),R(n,2k+1))$
then $\gamma_0 M \in (0,1)$.
\end{lemma}

\begin{proof}
We have the formulas (Equation (5.113) in \cite{Lev})
\[ \gamma_0=\frac{T_k(s,1)}{T_k(-1,-1)T_k(s,1)-T_k(-1,1)T_k(s,-1)} \]
and (the equation in the last line of page 488 in \cite{Lev})
\[ M=L_{2k}(n,s)=\frac{T_k(1,1)T_k(s,-1)-T_k(1,-1)T_k(s,1)}{T_k(s,-1)}. \]
Simple algebraic manipulations then show
\[ \gamma_0 M= \frac{T-A(s)}{T-1/A(s)}, \]
where $A(s)=T_k(s,1)/T_k(s,-1)$ as in \cite{BD} and $T=T_k(1,1)/T_k(1,-1)$.
Moreover, we have
\[ A(s)=T \cdot \frac{Q_k^{(1,0,n,2)}(s)}{Q_k^{(0,1,n,2)}(s)} \]
from \cite[Lemma 5.24]{Lev}, where $Q_k^{(1,0,n,2)}(s)>0$ and $Q_k^{(0,1,n,2)}(s)<0$ for every
$s \in \left(t_k^{1,0},t_k^{1,1}\right)$ (see Lemmas 5.29 and 5.30 in \cite{Lev}).
Therefore the signs of $A(s)$ and $T$ are opposite. We conclude that
\[  \gamma_0 M= \frac{|T|+|A(s)|}{|T|+1/|A(s)|}. \]
By Lemma 5.31 from \cite{Lev} the ratio $\frac{Q_k^{(1,0,n,2)}(s)}{Q_k^{(0,1,n,2)}(s)}$ is
decreasing in $s$ in the interval $\left(t_k^{1,0},t_k^{1,1}\right)$. Therefore
$|A(s)|$ is increasing in $s \in \left(t_k^{1,0},t_k^{1,1}\right)$.
Since $\gamma_0 M=0$ and 1 for $s=t_k^{1,0}$ and $t_k^{1,1}$, respectively,
we obtain that $\gamma_0 M$ increases from 0 to 1 when $s$ increases from $s=t_k^{1,0}$ to $t_k^{1,1}$.
\end{proof}

\medskip

\begin{remark}
\label{rem_antipodal}
We used in the proof the fact that the space $\mathbb{H}(n,2)$ is antipodal (i.e. for
every $x \in \mathbb{H}(n,2)$ there exists a unique antipodal point $y \in \mathbb{H}(n,2)$ such that
$d(x,y)=n \iff \langle x,y \rangle =-1$). Indeed, it follows from the
definition of the Krawtchouk polynomials for $q=2$
that the kernel $T_k(u,v)$ is symmetric (we need $T_k(1,1)=T_k(-1,-1)$ and $T_k(1,-1)=T_k(-1,1)$
that is not true for $q \geq 3$).
\end{remark}

\medskip

Next, we focus on utilizing estimates on the quantities $\ell(n,M,2k)$ and $s(n,M,2k)$ defined in Subsection 2.6 to improve our energy bounds.
\begin{lemma}
\label{Lem4.2}
(a part of Lemma 4.1 in \cite{BBD}) Let $C \subset \mathbb{H}(n,2)$ be a $(2k)$-design of cardinality $M \in (R(n,2k),R(n,2k+1))$
and $\xi$ be the smallest root of the equation $f(t)=\gamma_0 Mf(-1)$, where $f(t)=(t-\beta_1)^2\ldots(t-\beta_k)^2$.
Then $\ell(n,M,2k) \geq \xi$.
\end{lemma}

\medskip

Since $\gamma_0 M<1$ from Lemma \ref{Lem4.1} implies $f(\xi)=\gamma_0 Mf(-1)<f(-1)$ in Lemma \ref{Lem4.2}, it follows
that $-1<\xi \leq \ell(n,M,2k)$ for every $M \in (R(n,2k),R(n,2k+1))$.
We now modify Theorem \ref{thm 2} to use this fact.

\medskip

\begin{theorem}
\label{Thm_strict1} Let  $n$ and $h$ be as in Theorem \ref{thm 1}, $q=2$, $\tau=2k$ and $M \geq R(n,2k)$ be fixed.
Let $f(t)$ be a real polynomial that satisfies {\rm (A2$^\prime$)} and

{\rm (A1$^{\prime\prime}$)} $f(t) \leq h(t)$ for every $t \in T_n \cap [\ell(n,M,2k),s(n,M,2k)]$.

Then ${\mathcal L}(n,M,2k;h) \geq f_0M-f(1)$.
\end{theorem}

\medskip

It follows that Theorem \ref{Thm_strict1} gives strict improvements for the even branch of the  bound \eqref{bound_odd_designs}
in the whole range $M \in (R(n,2k),R(n,2k+1))$.

\begin{theorem}
\label{Thm_strict2} Let $\ell:=\ell(n,M,2k)$ and $G(t)$ be the Hermite interpolant of $h(t)$
\[ G(\ell)=h(\ell), \ G(\beta_i)=h(\beta_i), \ G^\prime(\beta_i)=h^\prime(\beta_i), \ i=1,\ldots,k, \]
$G(t)=\sum_{i=0}^{2k} G_iQ_i^{(n,2)}(t)$. Then
$\mathcal{L}(n,M,2k;h) \geq G_0M-G(1))>M\sum_{i=0}^k \gamma_i h(\beta_i)$.
\end{theorem}

\medskip

\begin{proof}
It follows from Theorem \ref{Thm_strict1} with $G(t)$ that $\mathcal{L}(n,M,2k;h) \geq G_0M-G(1))$.
The degree of $G(t)$ is $2k$ and we can apply the quadrature formula \eqref{quad_f0_beti}.
We obtain
\begin{eqnarray*}
G_0M-G(1) &=& M\sum_{i=0}^k \gamma_i G(\beta_i)=M\left(\gamma_0G(-1)+\sum_{i=1}^k \gamma_i h(\beta_i)\right) \\
&=& M\left(\gamma_0\left(G(-1)-h(-1)\right)+\sum_{i=0}^k \gamma_i h(\beta_i)\right) \\
&>& M \sum_{i=0}^k \gamma_i h(\beta_i),
\end{eqnarray*}
(the inequality $G(-1)>h(-1)$ follows from the interpolation since $-1<\ell(n,M,2k)$). \end{proof}

\medskip

We show as example how the better linear programming on subintervals works in the binary ($q=2$) case for $\tau=2$.
We first derive lower bounds on the quantity $\ell(n,M,2)$. Note that the approach here
is different than the one in Lemma \ref{Lem4.2}.

Let $C \subset \mathbb{H}(n,2)$ be a 2-design of minimum distance $d$ and cardinality $M=|C| \in (R(n,2),R(n,3))=(n+1,2n)$.
Since the space $\mathbb{H}(n,2)$ is antipodal, $C$ does not possess pairs of antipodal points \cite[Lemma 6.1]{Godsil},
\cite[Theorem 5.5(iv)]{BBD}. Let $x,y \in C$ be at maximum possible distance $d(x,y)=\tilde{d}<n$.

\medskip

\begin{lemma}
\label{bound_l_hamming}
We have
\begin{equation}
\label{l-bound_even}
\ell(n,M,2) \geq 1-\sqrt{\frac{2M}{n}}
\end{equation}
for all 2-designs $C \subset \mathbb{H}(n,2)$ with even $n-\tilde{d}$, and
\begin{equation}
\label{l-bound_odd}
\ell(n,M,2) \geq 1 - \frac{\sqrt{2(nM-2)}}{n}
\end{equation}
for all 2-designs $C \subset \mathbb{H}(n,2)$ with odd $n-\tilde{d} \geq 3$.
\end{lemma}

\begin{proof}
If $n-\tilde{d}=2d^\prime$ is even, we consider a point $u \in \mathbb{H}(n,2)$ such that\footnote{Here the point $-x$ is the unique point in $\mathbb{H}(n,2)$ such that $d(x,-x)=n$.} $d(-x,u)=d(y,u)=d^\prime$, and if
$n-\tilde{d}=2d^\prime+1 \geq 3$ is odd, we take a point $v \in \mathbb{H}(n,2)$ such that $d(-x,v)=d(y,v)-1=d^\prime$.
In both cases we apply \eqref{defin_f.3} from Definition \ref{def_des_alg} with the polynomial $f(t)=t^2$.

In the even case we have
\[  \frac{M}{n} =f_0|C|=\sum_{z \in C} f(\langle z,u\rangle) \geq 2f(\langle u,x \rangle)=2\left(1-\frac{2d^\prime}{n}\right)^2 \]
Then $2d^\prime \geq n-\sqrt{\frac{Mn}{2}}$ whence we get \eqref{l-bound_even}:
\[ \langle x,y \rangle =-\langle -x,y \rangle = \frac{4d^\prime}{n}-1 \geq 1-\sqrt{\frac{2M}{n}}. \]

The bound \eqref{l-bound_odd} in the case of odd $n-\tilde{d} \geq 3$ similarly follows by using
\[ f_0|C|=\sum_{z \in C} f(\langle z,v\rangle) \geq f(\langle v,x \rangle)+f(\langle v,y \rangle)
=\left(1-\frac{2d^\prime}{n}\right)^2+\left(1-\frac{2d^\prime+2}{n}\right)^2. \]
\end{proof}

\medskip

We can now find the optimal bound for $\mathcal{L}(n,M,2;h)$ for second degree polynomials which
satisfy $f(t) \leq h(t)$ for every $t \in [\ell,1)$, where $\ell$ is the lower bound for $\ell(n,M,2)$ from
Lemma \ref{bound_l_hamming}.

\medskip

\begin{theorem}
\label{bound_l2_hamming}
Let $q=2$, $n$, $M \in [R(n,2),R(n,3)]=[n+1,2n]$ and $h$ be fixed.
Let $\ell$ be the lower bound for $\ell(n,M,2)$ from Lemma \ref{bound_l_hamming}.
Then
\begin{equation}
\label{2lower}
\mathcal{L}(n,M,2;h) \geq \frac{n(M\ell+1-\ell)^2h(a_0)+M(M-n-1)h(\ell)}{M(1+n\ell^2)-n(1-\ell)^2},
\end{equation}
where $a_0=\frac{n(1-\ell)-M}{n(M\ell+1-\ell)}$.
\end{theorem}

\medskip

\begin{proof}
The second degree polynomial which graph passes
through the point $(\ell,h(\ell))$ and touches the graph of $h(t)$ at the point $(a_0,h(a_0))$
satisfies the conditions of Theorem~\ref{Thm_strict1} and gives the desired bound.
\end{proof}

\medskip

\begin{corollary}
\label{Cor5.10} If $q=2$ and $M/n \to \xi$, $\xi \in (1,2)$, as $n$ and $M$
tend to infinity simultaneously, then
\begin{equation}
\label{2lower_asymp}
\liminf_{n \to \infty} \frac{\mathcal{L}(n,M,2;h)}{n} \geq h(0)\xi .
\end{equation}
\end{corollary}

\medskip

\begin{proof}
The asymptotic of the bounds from Lemma \ref{bound_l_hamming} is $1-\sqrt{2\xi}:=\ell$. We plug this
in (\ref{2lower}) to obtain \eqref{2lower_asymp}.
\end{proof}

\section{Upper bounds for ${\mathcal U}(n,M,\tau;h)$}

To apply Theorem \ref{thm 4} we need upper bounds on $s(n,M,\tau)$. Such bounds can be obtained
as in Lemma \ref{Lem4.2} (see \cite[Lemma 4.1]{BBD}). We apply here different approach analogous
to Lemma \ref{bound_l_hamming} in the case $q=2$.

\begin{lemma}
\label{bound_u_hamming}
Let $C \subset \mathbb{H}(n,2)$ be a 2-design of minimum distance $d$ and cardinality $M=|C| \in (R(n,2),R(n,3))=(n+1,2n)$.
Then
\begin{equation}
\label{u-bound_even}
s(n,M,2) \leq -1 + \sqrt{\frac{2(M-2)}{n}}
\end{equation}
if $d$ is even, and
\begin{equation}
\label{u-bound_odd}
s(n,M,2) \leq -1 + \frac{1}{n}\sqrt{\frac{2(M-2)(nM-2)}{M}}
\end{equation}
if $d \geq 3$ is odd.
\end{lemma}

\begin{proof}
Let $x,y \in C$ be such that $d(x,y)=d$. If $d=2d^\prime$ is even, we consider
a point $u \in \mathbb{H}(n,2)$ such that $d(x,u)=d(y,u)=d^\prime$, and if
$d=2d^\prime+1$ is odd, we take a point $v \in \mathbb{H}(n,2)$ such that $d(x,v)=d(y,v)-1=d^\prime$.
In both cases we  apply \eqref{defin_f.3} from Definition \ref{def_des_alg} with a polynomial $f(t)=(t-a)^2$, where $a \in [-1,\beta_1]$
will be chosen to give the best possible upper bound on $s(n,M,2)$. Note that $\beta_1=s=-\frac{2n-M}{n(M-2)}$
in this case.

In the even case we have
\[ f_0|C|=\sum_{z \in C} f(\langle z,u\rangle) \geq 2f(\langle u,x \rangle) \]
for any $a<1-\frac{d}{n}$, whence $M(a^2+1/n) \geq 2\left(\langle u,x \rangle-a\right)^2$ (we used $f_0=a^2 + 1/n$).
The optimization over $a$ gives
\[ \langle u,x \rangle =1-\frac{d}{n} \leq \sqrt{\frac{M-2}{2n}} \]
(attained for $a= -\sqrt{\frac{2}{n(M-2)}}$). Then the equality
$\langle x,y \rangle =2\langle u,x \rangle -1$ implies \eqref{u-bound_even}.

The bound \eqref{u-bound_odd} in the case of odd $d$ similarly follows by using
\[ f_0|C|=\sum_{z \in C} f(\langle z,v\rangle) \geq f(\langle v,x \rangle)+f(\langle v,y \rangle)
=f\left(1-\frac{2d^\prime}{n}\right)+f\left(1-\frac{2d^\prime+2}{n}\right), \]
for any $a<1-\frac{d+1}{n}$.
\end{proof}

\medskip

The bounds for $\ell(n,M,2)$ and $s(n,M,2)$ from Lemmas \ref{bound_l_hamming} and \ref{bound_u_hamming}, respectively,
imply an easy upper bound on ${\mathcal U}(n,M,2;h)$ by Theorem \ref{thm 4} with a first degree polynomial.

\medskip

\begin{theorem}
\label{Thm6.2.} Let $q=2$, $n$, $M \in [R(n,2),R(n,3)]=[n+1,2n]$ and $h$ be fixed.
Let $\ell$ and $s$ be the lower and upper bounds for
$\ell(n,M,2)$ and $s(n,M,2)$ from Lemmas \ref{bound_l_hamming} and \ref{bound_u_hamming}, respectively.
Then
\begin{equation}
\label{2upper}
\mathcal{U}(n,M,2;h) \leq \frac{(M-1)(sh(\ell)-\ell h(s))+h(\ell)-h(s)}{s-\ell}.
\end{equation}
\end{theorem}

\medskip

\begin{proof}
The first degree polynomial which graph passes
through the points $(\ell,h(\ell))$ and $(s,h(s))$
satisfies the conditions of Theorem~\ref{thm 4} and gives the desired bound.
\end{proof}

\medskip

\begin{corollary}\label{Cor6.3} If $q=2$ and $n$ and $M=\xi n$, $\xi \in (1,2)$ is constant,
tend simultaneously to infinity, then
\begin{equation}
\label{2upper_asymp}
\mathcal{U}(n,M,2;h) \leq c_1n+c_2+o(1),
\end{equation}
where $c_1=\frac{\xi[(\sqrt{2\xi}-1)h(1-\sqrt{2\xi})-(1-\sqrt{2\xi})h(\sqrt{2\xi}-1)]}{2(\sqrt{2\xi}-1)}$ and
$c_2=\frac{(2-\sqrt{2\xi})h(1-\sqrt{2\xi})-\sqrt{2\xi}h(\sqrt{2\xi}-1)}{2(\sqrt{2\xi}-1)}$.
\end{corollary}

\medskip

\begin{proof}
The asymptotics of the bounds from Lemmas \ref{bound_l_hamming} and \ref{bound_u_hamming}
are $1-\sqrt{2\xi}:=\ell$ and $-1+\sqrt{2\xi}:=s$, respectively. We plug these in (\ref{2upper}) to obtain
\eqref{2upper_asymp}.
\end{proof}

\medskip

\begin{remark}
\label{rem_strip}
Theorems \ref{bound_l2_hamming} and \ref{Thm6.2.} (or Corollaries \ref{Cor5.10} and \ref{Cor6.3}, respectively)
give a strip where all 2-designs in $\mathbb{H}(n,2)$ have their energies. Such strips can be
obtained for higher strengths $\tau$ and in other spaces $\mathbb{H}(n,q)$ similar to \cite[Theorem 3.7]{BDHSS2}.
\end{remark}

\section{Asymptotics in the binary case}

For the binary case $\mathbb{H}(n,2)$,   we derive   asymptotics    for the lower bound in \eqref{bound_codes_odd}  as $n$ and the cardinality of codes goes to infinity. Note that in this case  that the Rao bounds satisfy
 $R(n,2k)=\sum_{j=0}^k\binom{n}{j}=\frac{n^k}{k!} +O(n^{k-1})$ and $R(n,2k-1)=\sum_{j=0}^{k-1}\binom{n-1}{j}=\frac{2n^{k-1}}{(k-1)!}+O(n^{k-2})$ as $n\to\infty$ and $k$ fixed. We consider sequence of codes of cardinalities $(M_n)$ satisfying
 $M_n\in I_{\tau}=(R(n,\tau),R(n,\tau+1))$ for $n=1, 2, 3, \ldots$ and
\begin{equation}
\label{asymp-1}
\lim_{n \to \infty} \frac{M_n}{n^{\lfloor \tau/2\rfloor}}=
\begin{cases} \frac{2}{(k-1)!}+\delta, & \tau=2k-1,\\[6pt]
\frac{1}{k!}+\delta, &  \tau=2k,
\end{cases}
\end{equation}
where $\delta \geq 0$.

\begin{lemma} \label{lem7.1} For  $i=0,1,2,\ldots$,   we have
\begin{equation}\label{Qasymp}
Q_i^{(n,2)}(t)=t^i + O(1/n),  \qquad t\in [-1,1],
\end{equation}
as $n\to \infty$ where the implied constants depend on $i$ but not $t$.

If, for $i=0,1,2,\ldots$, $P_i^{(n)}(t)$    is one of the adjacent polynomials
$Q_i^{(1,0,n,2)}(t)$, $Q_i^{(1,1,n,2)}(t)$, or $Q_i^{(0,1 ,n,2)}(t)$, then $$P_i^{(n)}(t)=t^i + O(1/n)$$ as $n\to \infty$.
\end{lemma}
\begin{proof}
Recalling from \eqref{Kraw} that $Q_i^{(n,2)}(t):=\frac{1}{r_i}K_i(z)=\frac{1}{r_i}K_i(n(1-t)/2)$   and using  \eqref{Kraw3term},
it follows that the polynomials $Q_i^{(n,2)}(t)$, $0\le i\le n$, satisfy the following three-term recurrence relation
\begin{equation}\label{Q3term}
\begin{split}
Q_i^{(n,2)}(t)&= \frac{nt}{n-i+1} \cdot Q_{i-1}^{(n,2)}(t)-\frac{i-1}{n-i+1}\cdot Q_{i-2}^{(n,2)}(t)\\
&=\left(t+O(1/n)\right)Q_{i-1}^{(n,2)}(t)+O(1/n)Q_{i-2}^{(n,2)}(t),
\end{split}
\end{equation}
where $Q_0^{(n,2)}(t)=1$ and $Q_1^{(n,2)}(t)=t$.
Then \eqref{Qasymp} follows by induction on $i$ using \eqref{Q3term} and the  form of
$Q_i^{(n,2)}(t)$ for $i=0$ and 1.

From  \eqref{adjacent}, we have
$$Q_i^{(1,0, n,2)}(t)=\frac{K_i^{(n-1,2)}(z-1)}{R(n,2i)}=\frac{r_i^{(n-1)}}{R(n,2i)}Q_i^{(n-1,2)}(t'),$$
where $\frac{n(1-t')}{2} =z'=z-1$.  Observing that $\frac{r_i^{(n-1)}}{R(n,2i)}=1+O(1/n)$ and that $t'=t+2/n$
shows that $Q_i^{(1,0, n,2)}(t)=Q_i^{(n,2)}(t) +O(1/n)$ as $n\to\infty$.  The result for the other two families of adjacent polynomials
follows similarly.
\end{proof}

We first deduce the limiting behavior of the nodes $\alpha_i$ and $\beta_i$ appearing in the lower bound.
Recall from Subsection~\ref{quadSec} that   the nodes $\alpha_i=\alpha_i(n,2k-1,M)$, $i=0,\ldots, k-1$,  are defined for  positive integers $n$, $k$, and $M$ satisfying $M>R(n,2k-1)$ and that the nodes $\beta_i=\beta_i(n,2k,M)$, $i=0,\ldots, k$,  are defined if $M>R(n,2k)$.

\medskip

\begin{lemma}
\label{asymp_alfi-beti} Let $\tau$ be a positive integer and $(M_n)$ be a sequence as in \eqref{asymp-1}.
If $\tau=2k-1$ for some integer $k$, then
\begin{equation}
\label{alpha0}
\begin{split}
 \lim_{n \to \infty} \alpha_0(n,2k-1,M_n) &=-1/(1+ \delta(k-1)! ), \text{ and }\\
\lim_{n \to \infty} \alpha_i(n,2k-1,M_n) &=0, \qquad  i=1,\ldots,k-1.
 \end{split}
 \end{equation}

If $\tau=2k$ for some integer $k$, then
\begin{equation}\label{beta1}
\lim_{n \to \infty} \beta_i(n,2k,M_n) =0, \qquad i=1,\ldots,k.
 \end{equation}
\end{lemma}

\begin{proof}
We first prove \eqref{beta1}.  Recall that $t=\beta_1,\ldots, \beta_k$ are  solutions  of
  \begin{equation}\label{MnL2k} M_n=L_{2k}(n,t). \end{equation}  Note that  $L_{2k}(n,t)=\left(1 - \frac{Q_{k-1}^{(1,1,n,2)}(t)}{Q_k^{(0,1,n,2)}(t)}\right)R(n,2k-1)$ and so
  any such solution satisfies
   $$\frac{M_n}{R(n,2k)} \cdot \frac{R(n,2k)}{R(n,2k-1)}=1 - \frac{Q_{k-1}^{(1,1,n,2)}(t)}{Q_k^{(0,1,n,2)}(t)}.$$
  Let $\epsilon>0$.  For $\epsilon<|s|\le 1$  Lemma~\ref{lem7.1} implies that  $\frac{Q_{k-1}^{(1,1,n,2)}(t)}{Q_k^{(0,1,n,2)}(t)}=t^{-1}+O(1/n)$ stays bounded as $n\to\infty$.    Since $\frac{R(n,2k)}{R(n,2k-1)} \sim \frac{n}{k}\to \infty$ and $\frac{M_n}{R(n,2k)}\ge 1$,  it follows that for $n$ large enough all solutions of \eqref{MnL2k} must satisfy $|t|<\epsilon$.  Hence, \eqref{beta1} holds.

The limits $\lim_{n \to \infty} \alpha_i =0$, $i=1,\ldots,k-1$, from \eqref{alpha0} follow from
the inequalities
\[ t_k^{1,1} >|\alpha_{k-1}|>|\alpha_1|>
|\alpha_{k-2}|>|\alpha_2|>\cdots \]
(cf. \cite[Corollary 3.9]{BD_pms}).
Now we use the Vieta formula
\[ \sum_{i=0}^{k-1} \alpha_i =\frac{(n-k)Q_k^{(1,0,n,2)}(s)}{nQ_{k-1}^{(1,0,n,2)}(s)}-\frac{k}{n} \]
(follows directly from \eqref{alfi}; can be seen in \cite[Lemma 4.3a)]{BD}) to conclude that
\[ \lim_{n \to \infty} \alpha_0=\lim_{n \to \infty} \frac{Q_k^{(1,0,n,2)}(s)}{Q_{k-1}^{(1,0,n,2)}(s)}.\]
The behavior of the ratio $Q_k^{(1,0,n,2)}(s)/Q_{k-1}^{(1,0,n,2)}(s)$ can be found by using the identities
(5.86) from \cite{Lev}
\[ M_n=L_{2k-1}(n,s)=\left(1-\frac{Q_{k-1}^{(1,0,n,2)}(s)}{Q_k^{(n,2)}(s)}\right)R(n,2k-2)=
\left(1-\frac{Q_k^{(1,0,n,2)}(s)}{Q_k^{(n,2)}(s)}\right)R(n,2k). \]
These imply
\[ \lim_{n \to \infty} \frac{Q_k^{(n,2)}(s)}{Q_{k-1}^{(1,0,n,2)}(s)}=-\frac{1}{1+ \delta(k-1)!},  \ \ \lim_{n \to \infty} \frac{Q_k^{(1,0,n,2)}(s)}{Q_k^{(n,2)}(s)}=1,
 \]
correspondingly. Therefore
\[ \lim_{n \to \infty} \alpha_0=\lim_{n \to \infty} \frac{Q_k^{(1,0,n,2)}(s)}{Q_k^{(n,2)}(s)}\cdot
\frac{Q_k^{(n,2)}(s)}{Q_{k-1}^{(1,0,n,2)}(s)}=-\frac{1}{1+ \delta(k-1)! }, \]
which completes the proof. \end{proof}

\begin{remark}
Another proof of \eqref{beta1} follows from
$\sum_{i=1}^{k-1} \beta_i=\frac{(n-k-1)Q_k^{(1,1,n,2)}(s)}{nQ_{k-1}^{(1,1,n,2)}(s)}$
(see \cite[Lemma 4.3b]{BD}, directly from the defining equation \eqref{beti})
and the behaviour of the ratio $Q_k^{(1,1,n,2)}(s)/Q_{k-1}^{(1,1,n,2)}(s)$ in
the interval $\mathcal{I}_{2k}$ -- it is non-positive, increasing, equal to zero in the right end $s=t_k^{1,1}$,
and tending to 0 as $n$ tends to infinity in the left end $s=t_k^{1,0}$.
\end{remark}

Recall that in the case $\tau=2k-1$ there are associated weights $\rho_i=\rho_i(n,2k-1,M_n)$, $i=0,\ldots, k-1$, such that the quadrature rule
 \eqref{quad_f0_alfi} holds for all polynomials of degree at most $\tau$ and, similarly, in the   case  $\tau=2k$ there are weights $\gamma_i=\gamma_i(n,2k-1,M_n)$, $i=0,\ldots, k$, such that \eqref{quad_f0_beti} holds for all polynomials of degree at most $\tau$.
 In view of Lemma \ref{asymp_alfi-beti} we need the asymptotic of $\rho_0(n,2k-1,M_n) M_n$ only.

\begin{lemma}
\label{asymp_rho0}
If $\tau=2k-1$ and $(M_n)$ is a sequence as in \eqref{asymp-1}, then $$\lim_{n \to \infty} \rho_0(n,2k-1,M_n) M_n =(1+\delta(k-1)!)^{2k-1}.$$
 \end{lemma}

\begin{proof}
This follows from the formula
\[ \rho_0(n,2k-1,M_n) M_n = - \frac{ (1-\alpha_1^2) (1-\alpha_2^2)\cdots(1-\alpha_{k-1}^2) }
                   { \alpha_0 (\alpha_0^2-\alpha_1^2)(\alpha_0^2-\alpha_2^2)
                   \cdots(\alpha_0^2-\alpha_{k-1}^2)} \]
(cf. \cite[Theorem 3.8]{BD_pms},
can be derived by setting $f(t)=t,t^3,\ldots,t^{2k-1}$ in \eqref{quad_f0_alfi}
and resolving the obtained linear system with respect to $\rho_0,\ldots,\rho_{k-1}$) and Lemma \ref{asymp_alfi-beti}.
\end{proof}

\medskip

\begin{theorem}
\label{Thm7.1}
Let $\tau$ be a positive integer and   $(M_n)$ be a sequence as in \eqref{asymp-1}. Then we have
\begin{equation}
\label{l-bound-a1}
\liminf_{n  \to \infty} \frac{\mathcal{E}(n,M_n;h)}{M_n} \geq h(0).
\end{equation}
\end{theorem}

\begin{proof} Let $\tau=2k-1$. We use Lemmas \ref{asymp_alfi-beti} and \ref{asymp_rho0}
    to calculate the odd branch of \eqref{bound_odd_designs} and \eqref{bound_codes_odd}:
    \begin{eqnarray*}
    \mathcal{E}(n,M_n;h) &\geq& M_n \sum_{i=0}^{k-1} \rho_i h(\alpha_i) \\
    &=& M_n \left(\rho_0h(\alpha_0)+h(0)\sum_{i=1}^{k-1} \rho_i+o(1)\right) \\
    &=& M_n \left(\rho_0(h(\alpha_0)-h(0))+h(0)\left(1-\frac{1}{M_n}+o(1)\right)\right) \\
    &=& h(0)M_n+c_3 +M_no(1),     \end{eqnarray*}      where $o(1)$ is a term that goes to 0 as $n\to\infty$ and     \[ c_3=\left(\left(1+\delta (k-1)!\right)^{2k-1}\right)\left(h\left(-\frac{1}{1+\delta (k-1)!}\right)-h(0)\right)-h(0). \]

    Similarly, in the even case we obtain
    \begin{eqnarray*}
    \mathcal{E}(n,M_n;h) &\ge& M_n\left(\gamma_0(h(-1)-h(0))+h(0)\left(1-\frac{1}{M_n}\right)+o(1)\right) \\
    &=& h(0)M_n+c_4+M_no(1),
    \end{eqnarray*}     where $c_4=\gamma_0M_n(h(-1)-h(0))-h(0)$ (here $\gamma_0 M_n \in (0,1)$, see Lemma \ref{Lem4.1}).  The above results establish \eqref{l-bound-a1}.  
        \end{proof}

\section{Examples}

Clearly, all codes which attain the Levenshtein bounds \eqref{L_bnd} achieve our bounds
\eqref{bound_odd_designs} and \eqref{bound_codes_odd} and are therefore universally optimal
(see Table 6.4 in \cite{Lev}). We show here two other examples where our bounds are close.
These examples are typical
for the behavior of our bounds.

There is a unique optimal (nonlinear) binary code of length 10 with 40 codewords
and minimum distance 4. We have $q=2$, $n=10$, $M=40$ and $\tau=3$.
Our bounds are very close, for example if $h=\frac{1}{5(1-t)}$, then
the actual energy is $8.125$, the universal bound is $\approx 8.0722$, the
pair-covering bound is $\approx 8.0857$, obtained by
\begin{eqnarray*}
f(t) &=& 0.111 t^3 + 0.200 t^2 + 0.205 t + 0.2 \\
     &=& 0.220 Q_0^{(10,2)}(t) + 0.236 Q_1^{(10,2)}(t) + 0.180 Q_2^{(10,2)}(t) + 0.080 Q_3^{(10,2)}(t)
\end{eqnarray*}
(here and below all numbers are truncated after the fourth digit).

There is a 5-design of 128 points in $H(9,2)$ (here $q=2$, $n=9$, $M=128$ and $\tau=5$).
For the potentials $h(t)=\left(\frac{2}{9(1-t)}\right)^s$  and $s = 0.1, 0.25, 0.5, 0.75, 1, 2.5$, respectively,
the actual energy is $\approx 109.861$, $88.593$, $62.284$, $44.143$, $31.546$,
$5.029$, the corresponding universal bound \eqref{bound_odd_designs} is $\approx 109.853$, $88.571$, $62.236$, $44.066$, $31.440$,
$4.828$, and the pair-covering bound is $\approx 109.858$, $88.584$, $62.264$, $44.111$, $31.503$,
$4.953$. All these bounds are valid for binary $(9,128)$ codes as well. For example,
the above pair-covering bound of $\approx 31.503$ (i.e. for $h(t)=\frac{2}{9(1-t)}$) is obtained by
the polynomial
\begin{eqnarray*}
f(t) &=& 0.183 t^5 + 0.345 t^4 + 0.284 t^3 + 0.216 t^2 + 0.216 t + 0.221 \\
     &=& 0.257 Q_0^{(9,2)}(t) + 0.330 Q_1^{(9,2)}(t) + 0.366 Q_2^{(9,2)}(t) + 0.306 Q_3^{(9,2)}(t) \\
     && + \, 0.159 Q_4^{(9,2)}(t) + 0.046 Q_5^{(9,2)}(t)
\end{eqnarray*}
that satisfies the condition (A2) as well. Here, the best lower bound is $31.525$ and
can be obtained by a polynomial of degree 9
\begin{eqnarray*}
f(t) &=& 0.540 t^9 - 1.041 t^7 + 0.773 t^5 + 0.345 t^4 + 0.171 t^3 + 0.210 t^2 + 0.222 t + 0.222 \\
     &=& 0.257 Q_0^{(9,2)}(t) + 0.330 Q_1^{(9,2)}(t) + 0.361 Q_2^{(9,2)}(t) + 0.296 Q_3^{(9,2)}(t) \\
     && + \, 0.159 Q_4^{(9,2)}(t) + 0.0394 Q_5^{(9,2)}(t) + 0.005 Q_9^{(9,2)}(t)
\end{eqnarray*}
that also satisfies (A2).

{\bf Acknowledgement.} The authors thank to the anonymous referees for valuable remarks.

\end{document}